\documentclass[11pt]{article}
\usepackage{amsmath,amsthm,amssymb,float}
\theoremstyle{plain}
\newtheorem{theorem}{Theorem}[section]
\newtheorem{proposition}[theorem]{Proposition}

\newtheorem{corollary}[theorem]{Corollary}
\newtheorem{claim}[theorem]{Claim}

\theoremstyle{definition}
\newtheorem{definition}[theorem]{Definition}
\newtheorem{example}[theorem]{Example}
\theoremstyle{remark}
\newtheorem{remark}[theorem]{Remark}
\numberwithin{equation}{section}

\renewcommand{\eqref}[1]{(\ref{#1})}
\newcommand{\field}[1]{\mathbb{#1}}

\newcommand{\Z}{\field{Z}}
\newcommand{\K}{\field{K}}
\newcommand{\id}{{\mathrm{id}}}

\begin{document}
\title{Construction of Hopf algebroids}
\author{Yudai Otsuto\thanks{Department of Mathematics,
Faculty of Science, Hokkaido University, Sapporo 0600810, Japan; my.otsuto@math.sci.hokudai.ac.jp} 
\and Youichi Shibukawa\thanks{Department of Mathematics,
Faculty of Science, Hokkaido University, Sapporo 0600810, Japan; shibu@math.sci.hokudai.ac.jp}}
\date{}
\maketitle
\begin{abstract}
For arbitrary algebras $L$,
we construct Hopf algebroids $A_\sigma$ with base rings $L$
by means of $\sigma^{ab}_{cd}\in L$
satisfying suitable properties.
\end{abstract}
\section{Introduction}
The quantum group \cite{drinfeld,jimbo}
has produced a revival of interest 
in the Hopf algebra \cite{abe,montgomery,sweedler},
and much attention is now directed to
its generalization \cite{lu,maltsiniotis,ravenel,xu}.

After pioneering works \cite{hayashi1,hayashi2}
by Hayashi about face algebras,
also called weak Hopf algebras,
there are two generalizations of the Hopf algebra:
one is the $\times_L$-Hopf algebra \cite{schauenburg}
and the other is the Hopf algebroid \cite{bohm2005,bohm2004},
on which we focus in this paper.

Motivated by
the FRT construction of
a q-analogue of the function space \cite{etingof,faddeev},
one of the author constructed Hopf algebroids
$A_\sigma$ \cite{kamiyas2018,s2016,stakeuchi}
by means of
dynamical Yang-Baxter maps $\sigma$
\cite{kamiyas,matsumoto,matsumotos,matsushimi,
rump,s2005,s2007,s20101,s20102}.
This $\sigma$ is a set-theoretic
solution
to a version of the quantum dynamical
Yang-Baxter equation \cite{felder,gervais},
a generalization of the quantum Yang-Baxter equation
\cite{baxter72,baxter78,yang67,yang68}.

Let
$\K$ be an arbitrary field,
$H$ a nonempty and finite set,
and let $M_H(\K)$ denote
the $\K$-algebra consisting of maps from the set $H$ to $\K$.
In the above construction of Hopf algebroids \cite{s2016},
the algebra $A_\sigma$ we first obtained
is nothing but a weak Hopf algebra,
because
the following two conditions are equivalent
\cite[Section 5]{schauenburg2003}
by taking account of the fact that
the Hopf algebroid is a $\times_L$-Hopf algebra
\cite{bohm2004}:
(1)
$A_\sigma$ is a weak Hopf algebra;
(2)
the base ring $M_H(\K)$ of the Hopf algebroid $A_\sigma$
is separable with an idempotent Frobenius system.
Here,
for a commutative ring $k$,
we say that a 
$k$-algebra $L$ is separable with an idempotent Frobenius system
\cite[Section 3]{schauenburg2003},
iff 
there exist a $k$-linear map $\phi: L\to k$
and
an element $e=\sum e^{(1)}\otimes e^{(2)}\in L\otimes_k L$
satisfying that
$l=\sum\phi(le^{(1)})e^{(2)}=\sum e^{(1)}\phi(e^{(2)}l)$
for any $l\in L$
and that
$\sum e^{(1)}e^{(2)}=1_L$.
For construction of the Hopf algebroid,
we
need an extra process \cite{s2016},
in which we change from the finite set $H$ to an infinite one.
This change implies that the base ring $M_H(\K)$
is not separable
with an idempotent Frobenius system.

The purpose of this paper is to give a simpler way to
construct 
Hopf algebroids
that are not weak Hopf algebras,
even though the set $H$ is finite.

To achieve this purpose, we
tried to generalize 
$M_H(\K)$ to the $\K$-algebra
$M_H(R)$
consisting of maps from the set $H$ to
an arbitrary $\K$-algebra $R$ \cite{otsuto}.
If $R$ is not separable with
an idempotent Frobenius system,
then so is the algebra $M_H(R)$,
and 
the Hopf algebroid $A_\sigma$ 
with $M_H(R)$ is not a weak Hopf algebra
as a result
(See also Section 7).

The aim of this paper is to generalize the 
algebra $M_H(R)$ to an arbitrary algebra $L$
in this construction;
and we will consequently clarify properties of $\sigma$
that can
produce the Hopf algebroid $A_\sigma$.

The organization of this paper is as follows.
Section 2 presents a left bialgebroid $A_\sigma$.
This algebra $A_\sigma$ is also a right bialgebroid,
which is proved in Section 3.
In Section 4, we construct an anti-automorphism $S$ on 
the algebra $A_\sigma$,
which implies that $A_\sigma$ is a Hopf algebroid in Section 5.
Section 6 deals with a sufficient condition for
the existence of the anti-automorphism $S$ on $A_\sigma$.
In Section 7, we provide with examples of Hopf algebroids
$A_\sigma$
by means of the sufficient condition in Section 6.
\section{Left bialgebroid $A_\sigma$}
In this section, we introduce a left bialgebroid
$A_\sigma$, which is a main subject of this paper.

Let $A$ and $L$ be associative unital rings,
and let
$s_L: L\to A$
and $t_L: L^{op}\to A$ be ring homomorphisms
satisfying
\begin{equation}\label{eq:leftbi:commute}
s_L(l)t_L(l')=t_L(l')s_L(l)
\quad(\forall l, l'\in L).
\end{equation}
Here, $L^{op}$ is the opposite ring of $L$.
These two homomorphisms give an $L$-bimodule structure in $A$,
which is denoted by
${}_LA_L$,
through the following action:
\begin{equation}\label{eq:leftbi:action}
l\cdot a\cdot l'=s_L(l)t_L(l')a
\quad(l, l'\in L, a\in A).
\end{equation}

For $f\in L$,
the notations $\rho_l(f)$
and $\rho_r(f)$ respectively mean the left 
and the right multiplications
by $f$:
\begin{equation}\label{eq:leftbi:rholr}
\rho_l(f): L\ni g\mapsto fg\in L;
\quad\rho_r(f): L\ni g\mapsto gf\in L.
\end{equation}
These are elements of $\mathrm{End}_k(L)$
and make $L$ an $L$-bimodule:
\[
l\cdot f\cdot l'=\rho_l(l)\rho_r(l')f
\quad(l, l', f\in L).
\]
\begin{definition}\label{def:leftbi:leftbialgebroid}
Let $({}_LA_L, \Delta_L, \pi_L)$
be a comonoid in the tensor category of $L$-bimodules.
A sextuplet $A_L=(A, L, s_L, t_L, \Delta_L, \pi_L)$
is called a left bialgebroid,
iff
{\allowdisplaybreaks
\begin{align}
\label{eq:leftbi:Deltast}
&a_{(1)}t_L(l)\otimes a_{(2)}=a_{(1)}\otimes a_{(2)}s_L(l),
\\\label{eq:leftbi:Delta1A}
&\Delta_L(1_A)=1_A\otimes 1_A,
\\\label{eq:leftbi:Delta}
&\Delta_L(ab)=\Delta_L(a)\Delta_L(b),
\\\label{eq:leftbi:pi1A}
&\pi_L(1_A)=1_L,
\\\label{eq:leftbi:pist}
&\pi_L(as_L(\pi_L(b)))=\pi_L(ab)=\pi_L(at_L(\pi_L(b)))
\end{align}
}%
for any $l\in L$
and $a, b\in A$.
Here,
$1_A$ is the unit element of the ring $A$
and $1_L$ is that of the ring $L$.
We use Sweedler's notation in \eqref{eq:leftbi:Deltast}:
$\Delta_L(a)=a_{(1)}\otimes a_{(2)}\in A\otimes_LA$.
The right-hand-side of \eqref{eq:leftbi:Delta}
is well defined because of \eqref{eq:leftbi:Deltast}.
\end{definition}
The left bialgebroid is also called 
Takeuchi's $\times_L$-bialgebra \cite{takeuchi}.

From now on,
the symbols $k$ and $L$ respectively denote
a commutative ring with the unit $1_k$
and
a $k$-algebra with the unit $1_L$.
That is to say,
$L$ is a $k$-module with an associative multiplication that is
bilinear and has the unit element $1_L$.
The letter 
$G$ means a group,
and,
for any $\alpha\in G$,
let $T_\alpha: L\to L$ be
a $k$-algebra automorphism
satisfying 
\begin{equation}\label{eq:leftbi:translation}
T_\alpha\circ T_{\alpha^{-1}}=\id_L
\quad(\forall\alpha\in G).
\end{equation}
We will denote by $\deg$ a map from a finite set $X$ to
the group $G$.

Let 
$Gen$ denote the set
$(L\otimes_kL^{op})\coprod\{ L_{ab}: a, b\in X\}\coprod
\{ (L^{-1})_{ab}: a, b\in X\}$.
Here, 
$L_{ab}$
and $(L^{-1})_{ab}$ are indeterminates,
and
$L^{op}$ is the opposite algebra of $L$.
$\langle Gen\rangle$ means
the monoid consisting of all (finite) words including
the empty word $\emptyset$ whose set of alphabets is $Gen$.
The binary operation of this monoid is a concatenation of words.
Let $\sigma^{ab}_{cd}\in L$
$(a, b, c, d\in X)$
and
we denote by $I_\sigma$
the two-sided ideal of the free $k$-algebra 
$k\langle Gen\rangle=\oplus_{w\in\langle Gen\rangle}
kw$ whose generators are the following.
\begin{enumerate}
\item
$\xi+\xi'-(\xi+\xi')$,
$c\xi-(c\xi)$,
$\xi\xi'-(\xi\xi')$
$(\forall c\in k, \xi, \xi'\in L\otimes_k L^{op})$.

Here, the symbol $+$ in $\xi+\xi'$ means the addition in the algebra 
$k\langle Gen\rangle$,
while the symbol $+$ in $(\xi+\xi')(\in Gen)$ is the addition in the algebra
$L\otimes_k L^{op}$.
The notations of the scalar products and products in the other generators 
are similar.
\item
$\displaystyle
\sum_{c\in X}L_{ac}(L^{-1})_{cb}-\delta_{ab}\emptyset$,
$\displaystyle
\sum_{c\in X}(L^{-1})_{ac}L_{cb}-\delta_{ab}\emptyset$
$(\forall a, b\in X)$.

Here, $\delta_{ab}$ denotes  
Kronecker's delta symbol;
$\delta_{ab}=\begin{cases}1_k&\mbox{if $a=b$};\\0_k&\mbox{otherwise.}
\end{cases}$
\item
$(T_{\mathrm{deg}(a)}(f)\otimes 1_L)L_{ab}-L_{ab}(f\otimes 1_L)$,

$(1_L\otimes T_{\mathrm{deg}(b)}(f))L_{ab}-L_{ab}(1_L\otimes f)$,

$(f\otimes 1_L)(L^{-1})_{ab}-(L^{-1})_{ab}
(T_{\mathrm{deg}(b)}(f)\otimes 1_L)$,

$(1_L\otimes f)
(L^{-1})_{ab}-(L^{-1})_{ab}
(1_L\otimes T_{\mathrm{deg}(a)}(f))$
$(\forall f\in L(=L^{op}), a, b\in X)$.
\item 
$\sum_{x, y\in X}(\sigma^{xy}_{ac}\otimes1_L)L_{yd}L_{xb}
-
\sum_{x, y\in X}(1_L\otimes\sigma^{bd}_{xy})L_{cy}L_{ax}$
$(\forall a, b, c, d\in X)$.
\item
$\emptyset-1_L\otimes 1_L$.
\end{enumerate}

Let us denote by $A_\sigma$
the quotient
\begin{equation}\label{eq:leftbi:Asigma}
A_\sigma=k\langle Gen\rangle/I_\sigma.
\end{equation}
\begin{theorem}\label{thm:leftbi:leftbialg}
If
the elements $\sigma^{ab}_{cd}\in L$
satisfy
\begin{equation}\label{eq:leftbi:rhoT}
\rho_l(\sigma^{bd}_{ac})\circ T_{\deg(d)}\circ T_{\deg(b)}
=\rho_r(\sigma^{bd}_{ac})\circ T_{\deg(c)}\circ T_{\deg(a)}
\end{equation}
for any $a, b, c, d\in X$,
then
the quotient $A_\sigma$ $\eqref{eq:leftbi:Asigma}$
is a left bialgebroid.
\end{theorem}
The maps $s_L: L\to A_\sigma$
and
$t_L: L^{op}\to A_\sigma$ are defined as follows:
\begin{align}
&s_L: L\ni l\mapsto l\otimes 1_L+I_\sigma\in A_\sigma;
\label{eq:leftbi:sL}\\\label{eq:leftbi:tL}
&t_L: L^{op}\ni l\mapsto 1_L\otimes l+I_\sigma\in A_\sigma.
\end{align}
These are $k$-algebra homomorphisms
satisfying \eqref{eq:leftbi:commute},
and $A_\sigma$ is consequently an $L$-bimodule
with the action \eqref{eq:leftbi:action}.

In order to define the map $\Delta_L$,
we need a $k$-algebra homomorphism
$\overline{\Delta}: k\langle Gen\rangle\to A_\sigma\otimes_kA_\sigma$
whose definition on the generators is as follows:
\begin{align}\nonumber
&\overline{\Delta}(\xi)=(s_L\otimes_kt_L)(\xi)
\quad(\xi\in L\otimes_kL^{op});
\\\nonumber
&\overline{\Delta}(L_{ab})=\sum_{c\in X}
L_{ac}+I_\sigma\otimes L_{cb}+I_\sigma
\quad(a, b\in X);
\\\label{eq:leftbi:DeltaL}
&\overline{\Delta}((L^{-1})_{ab})=\sum_{c\in X}
(L^{-1})_{cb}+I_\sigma\otimes(L^{-1})_{ac}+I_\sigma
\quad(a, b\in X).
\end{align}
We write $I_2$ for the right ideal of $A_\sigma\otimes_kA_\sigma$
whose generators are
$t_L(l)\otimes 1_{A_\sigma}-1_{A_\sigma}\otimes s_L(l)$
$(\forall l\in L)$.
\begin{proposition}\label{prop:leftbi:rightideal}
$I_2$ is a $k$-module.
If $a\in I_\sigma$ is any generator $(1)$--$(5)$
of two-sided ideal $I_\sigma$, then
$\overline{\Delta}(a)\in I_2$.
In addition,
$\overline{\Delta}(k\langle Gen\rangle)I_2\subset I_2$.
\end{proposition}
From this proposition,
$\overline{\Delta}(I_\sigma)\subset I_2$,
which induces a $k$-module homomorphism
$\widetilde{\Delta}: A_\sigma\to(A_\sigma\otimes_k
A_\sigma)/I_2$.
By taking account of the fact that
$(A_\sigma\otimes_kA_\sigma)/I_2\cong A_\sigma\otimes_LA_\sigma$
as $\Z$-modules,
this $\widetilde{\Delta}$
induces 
the $\Z$-module homomorphism
$\Delta_L: A_\sigma\to A_\sigma\otimes_LA_\sigma$,
which is also an $L$-bimodule homomorphism.
\begin{proof}[Proof of Proposition $\ref{prop:leftbi:rightideal}$]
We give the proof only for that $\overline{\Delta}(k\langle Gen\rangle)I_2\subset I_2$.
It is sufficient to show that
$\overline{\Delta}(v)(t_L(l)\otimes 1_{A_\sigma}
-1_{A_\sigma}\otimes s_L(l))\in I_2$
for any $v\in Gen$ and $l\in L$,
since $I_2$ is a right ideal.

If $v=(L^{-1})_{ab}$,
then, on account of \eqref{eq:leftbi:translation}
and the generators (3) of the ideal $I_\sigma$,
\begin{align*}
&\overline{\Delta}(v)(t_L(l)\otimes 1_{A_\sigma}
-1_{A_\sigma}\otimes s_L(l))
\\
=&
\sum_{c\in X}
(L^{-1})_{cb}(1_L\otimes l)+I_\sigma
\otimes(L^{-1})_{ac}+I_\sigma
\\
&-
\sum_{c\in X}
(L^{-1})_{cb}+I_\sigma
\otimes(L^{-1})_{ac}(l\otimes 1_L)+I_\sigma
\\
=&
\sum_{c\in X}
(1_L\otimes T_{\deg(c)^{-1}}(l))(L^{-1})_{cb}+I_\sigma
\otimes(L^{-1})_{ac}+I_\sigma
\\
&-
\sum_{c\in X}
(L^{-1})_{cb}+I_\sigma
\otimes(T_{\deg(c)^{-1}}(l)\otimes 1_L)(L^{-1})_{ac}+I_\sigma
\\
=&
\sum_{c\in X}
(t_L(T_{\deg(c)^{-1}}(l))\otimes 1_{A_\sigma}
-1_{A_\sigma}\otimes s_L(T_{\deg(c)^{-1}}(l)))\times
\\
&\quad\times((L^{-1})_{cb}+I_\sigma
\otimes(L^{-1})_{ac}+I_\sigma)
\\
\in& I_2.
\end{align*}

The proof is easy for the other $v\in Gen$.
\end{proof}

The next task is to define the map $\pi_L: A_\sigma\to L$.
For this purpose, we first construct the $k$-algebra homomorphism
$\overline{\varepsilon}: k\langle Gen\rangle\to\mathrm{End}_k(L)$,
which is defined on the generators as follows:
\[
\overline{\varepsilon}(L_{ab})=\delta_{ab}T_{\deg(a)},
\ 
\overline{\varepsilon}((L^{-1})_{ab})=\delta_{ab}T_{\deg(a)^{-1}}
\quad(a, b\in X),
\]
and $\overline{\varepsilon}$ on $L\otimes_kL^{op}$
is the (unique) $k$-linear map
satisfying
$\overline{\varepsilon}(l\otimes l')=\rho_l(l)\rho_r(l')$
$(l, l'\in L)$.
For $\rho_l$
and $\rho_r$, see \eqref{eq:leftbi:rholr}.
\begin{proposition}
$\overline{\varepsilon}(I_\sigma)=\{ 0\}$.
\end{proposition}
\begin{proof}
It suffices to prove that
$\overline{\varepsilon}(a)=0$
for any generator $a$ of the two-sided ideal $I_\sigma$.
We give the proof only for the case that $a$ is a generator (3)
or (4).

Because $T_\alpha$ is a $k$-algebra homomorphism
satisfying \eqref{eq:leftbi:translation},
$\overline{\varepsilon}(a)=0$
for the case that $a$ is a generator (3).

By virtue of \eqref{eq:leftbi:rhoT},
$\overline{\varepsilon}(a)=0$
for the case that $a$ is a generator (4).
\end{proof}
According to this proposition,
$\overline{\varepsilon}$
induces
the $k$-algebra homomorphism
$\varepsilon: A_\sigma\to\mathrm{End}_k(L)$,
and the definition of the map $\pi_L$ is that
\[
\pi_L: A_\sigma\ni a\mapsto \varepsilon(a)(1_L)\in L.
\]
This $\pi_L$ is an $L$-bimodule homomorphism.

The triplet
$(A_\sigma, \Delta_L, \pi_L)$ is a comonoid in
the tensor category of $L$-bimodules.
Moreover,
the sextuplet
$A_\sigma=(A_\sigma, L, s_L, t_L, \Delta_L, \pi_L)$
satisfies
\eqref{eq:leftbi:Deltast}--\eqref{eq:leftbi:pist}
in Definition
\ref{def:leftbi:leftbialgebroid},
and
$A_\sigma$ is thus a left bialgebroid,
which completes the proof of Theorem \ref{thm:leftbi:leftbialg}.
\section{Right bialgebroid $A_\sigma$}
In this section, 
we clarify the condition that makes the algebra $A_\sigma$
in the previous section to be a right bialgebroid,
and
show that 
this condition implies
\eqref{eq:leftbi:rhoT};
hence,
$A_\sigma$ is a left and right bialgebroid under the condition.

Let $A$ and $L'$ be associative unital rings,
and let
$s_{L'}: L'\to A$
and $t_{L'}: {L'}^{op}\to A$ be ring homomorphisms
satisfying
\begin{equation}\label{eq:rightbi:commute}
s_{L'}(r)t_{L'}(r')=t_{L'}(r')s_{L'}(r)
\quad(\forall r, r'\in {L'}).
\end{equation}
These two homomorphisms give an $L'$-bimodule structure in $A$,
which is denoted by
${}^{L'}\!\!A^{L'}$,
through the following action:
\begin{equation}\label{eq:rightbi:action}
r\cdot a\cdot r'=as_{L'}(r')t_{L'}(r)
\quad(r, r'\in L', a\in A).
\end{equation}
\begin{definition}\label{def:rightbi:rightbialgebroid}
Let $({}^{L'}\!\!A^{L'}, \Delta_{L'}, \pi_{L'})$
be a comonoid in the tensor category of $L'$-bimodules.
A sextuplet $A_{L'}=(A, L', s_{L'}, t_{L'}, \Delta_{L'}, \pi_{L'})$
is called a right bialgebroid,
iff
{\allowdisplaybreaks
\begin{align}
\label{eq:rightbi:Deltast}
&s_{L'}(r)a^{(1)}\otimes a^{(2)}=a^{(1)}\otimes t_{L'}(r)a^{(2)},
\\\label{eq:rightbi:Delta1A}
&\Delta_{L'}(1_A)=1_A\otimes 1_A,
\\\label{eq:rightbi:Delta}
&\Delta_{L'}(ab)=\Delta_{L'}(a)\Delta_{L'}(b),
\\\label{eq:rightbi:pi1A}
&\pi_{L'}(1_A)=1_{L'},
\\\label{eq:rightbi:pist}
&\pi_{L'}(s_{L'}(\pi_{L'}(a))b)=\pi_{L'}(ab)
=\pi_{L'}(t_{L'}(\pi_{L'}(a))b)
\end{align}
}%
for any $r\in L'$
and $a, b\in A$.
Here, $a^{(1)}\otimes a^{(2)}$ is Sweedler's notation:
$\Delta_{L'}(a)=a^{(1)}\otimes a^{(2)}$.
The right-hand-side of \eqref{eq:rightbi:Delta}
is well defined because of \eqref{eq:rightbi:Deltast}.
\end{definition}
\begin{theorem}\label{thm:rightbi:rightbialg}
If
the elements $\sigma^{ab}_{cd}\in L$
satisfy
\begin{equation}\label{eq:rightbi:rhoT}
T_{\deg(a)^{-1}}\circ T_{\deg(c)^{-1}}\circ\rho_l(\sigma^{bd}_{ac})
=T_{\deg(b)^{-1}}\circ T_{\deg(d)^{-1}}\circ\rho_r(\sigma^{bd}_{ac})
\end{equation}
for any $a, b, c, d\in X$,
then the algebra
$A_\sigma$ $\eqref{eq:leftbi:Asigma}$
is a right bialgebroid for $L'=L^{op}$
$($For $\rho_l$
and $\rho_r$, see $\eqref{eq:leftbi:rholr}$$)$.
\end{theorem}
The maps $s_L$ \eqref{eq:leftbi:sL}
and $t_L$
\eqref{eq:leftbi:tL}
define
$s_{L^{op}}: L^{op}\to A_\sigma$
and
$t_{L^{op}}: L\to A_\sigma$:
$s_{L^{op}}:=t_L$;
$t_{L^{op}}:=s_L$.
These are $k$-algebra homomorphisms
satisfying \eqref{eq:rightbi:commute},
and $A_\sigma$ is consequently an $L^{op}$-bimodule
with the action \eqref{eq:rightbi:action}.

In order to define the map $\Delta_{L^{op}}$,
we make use of the $k$-algebra homomorphism
$\overline{\Delta}: k\langle Gen\rangle\to A_\sigma\otimes_kA_\sigma$
in
\eqref{eq:leftbi:DeltaL}.
We write $I_2'$ for the left ideal of $A_\sigma\otimes_kA_\sigma$
whose generators are
$s_{L^{op}}(l)\otimes 1_{A_\sigma}-1_{A_\sigma}\otimes t_{L^{op}}(l)$
$(\forall l\in L^{op}(=L))$.
This $I_2'$ is a $k$-module
and 
$\overline{\Delta}(I_\sigma)\subset I_2'$,
which induces a $k$-module homomorphism
$\widetilde{\Delta}': A_\sigma\to(A_\sigma\otimes_k
A_\sigma)/I_2'$.
The proof is similar to that of
Proposition \ref{prop:leftbi:rightideal}.
Because
$(A_\sigma\otimes_kA_\sigma)/I_2'\cong 
A_\sigma\otimes_{L^{op}}A_\sigma$
as $\Z$-modules,
this $\widetilde{\Delta}'$
implies 
the $\Z$-module homomorphism
$\Delta_{L^{op}}: A_\sigma\to A_\sigma\otimes_{L^{op}}A_\sigma$,
which is also an $L^{op}$-bimodule homomorphism.

The next task is to define a map $\pi_{L^{op}}: A_\sigma\to L^{op}$.
We construct the $k$-algebra anti-homomorphism
$\overline{\varepsilon}': k\langle Gen\rangle\to\mathrm{End}_k(L^{op})$
whose definition on the generators is as follows:
\[
\overline{\varepsilon}'(L_{ab})=\delta_{ab}T_{\deg(a)^{-1}},
\ 
\overline{\varepsilon}'((L^{-1})_{ab})=\delta_{ab}T_{\deg(a)}
\quad(a, b\in X),
\]
and $\overline{\varepsilon}'$ on $L\otimes_kL^{op}$
is the (unique) $k$-linear map
satisfying
$\overline{\varepsilon}'(l\otimes l')=\rho_l(l')\rho_r(l)$
$(l, l'\in L^{op}(=L))$.
According to the fact that 
$\overline{\varepsilon}'(I_\sigma)=\{ 0\}$,
$\overline{\varepsilon}'$
induces
the $k$-algebra anti-homomorphism
$\varepsilon': A_\sigma\to\mathrm{End}_k(L^{op})$,
and the definition of the map $\pi_{L^{op}}$ is the following.
\[
\pi_{L^{op}}: A_\sigma\ni a\mapsto \varepsilon'(a)(1_L)\in L^{op}.
\]
This $\pi_{L^{op}}$ is an $L^{op}$-bimodule homomorphism.

The triplet
$(A_\sigma, \Delta_{L^{op}}, \pi_{L^{op}})$ is a comonoid in
the tensor category of $L^{op}$-bimodules.
In addition,
the sextuplet
$A_\sigma=(A_\sigma, L^{op}, s_{L^{op}}, t_{L^{op}},
\Delta_{L^{op}}, \pi_{L^{op}})$
satisfies
\eqref{eq:rightbi:Deltast}--\eqref{eq:rightbi:pist}
in Definition
\ref{def:rightbi:rightbialgebroid},
and
$A_\sigma$ is hence a right bialgebroid.
\begin{proposition}\label{prop:rightbi:rightleft}
This right bialgebroid $A_\sigma$ satisfying
$\eqref{eq:rightbi:rhoT}$ is also a left bialgebroid.
\end{proposition}
\begin{proof}
It suffices to prove that
\eqref{eq:rightbi:rhoT}
implies \eqref{eq:leftbi:rhoT}.
From \eqref{eq:rightbi:rhoT},
\[
T_{\deg(a)^{-1}}\circ T_{\deg(c)^{-1}}
\circ\rho_l(\sigma^{bd}_{ac})(1_L)
=T_{\deg(b)^{-1}}\circ T_{\deg(d)^{-1}}
\circ\rho_r(\sigma^{bd}_{ac})(1_L),
\]
which is exactly the same as
\begin{equation}\label{eq:rigid:TT}
T_{\deg(a)^{-1}}\circ T_{\deg(c)^{-1}}(\sigma^{bd}_{ac})
=T_{\deg(b)^{-1}}\circ T_{\deg(d)^{-1}}(\sigma^{bd}_{ac}).
\end{equation}
By means of \eqref{eq:rightbi:rhoT}
and \eqref{eq:rigid:TT},
the left-hand-side of \eqref{eq:leftbi:rhoT} is
{\allowdisplaybreaks
\begin{align*}
&
T_{\deg(c)}\circ T_{\deg(a)}\circ T_{\deg(a)^{-1}}
\circ T_{\deg(c)^{-1}}\circ\rho_l(\sigma^{bd}_{ac})
\circ T_{\deg(d)}\circ T_{\deg(b)}
\\
=&
T_{\deg(c)}\circ T_{\deg(a)}\circ T_{\deg(b)^{-1}}
\circ T_{\deg(d)^{-1}}\circ\rho_r(\sigma^{bd}_{ac})
\circ T_{\deg(d)}\circ T_{\deg(b)}
\\
=&
T_{\deg(c)}\circ T_{\deg(a)}\circ 
\rho_r(T_{\deg(b)^{-1}}
\circ T_{\deg(d)^{-1}}(\sigma^{bd}_{ac}))
\\
=&
\rho_r(T_{\deg(c)}\circ T_{\deg(a)}\circ T_{\deg(b)^{-1}}
\circ T_{\deg(d)^{-1}}(\sigma^{bd}_{ac}))
\circ T_{\deg(c)}\circ T_{\deg(a)}
\\
=&
\rho_r(\sigma^{bd}_{ac})
\circ T_{\deg(c)}\circ T_{\deg(a)}.
\end{align*}}%
This establishes the formula.
\end{proof}
\begin{remark}\label{rem:rightbi:remark}
We note that \eqref{eq:rightbi:rhoT} is equivalent to
\eqref{eq:leftbi:rhoT} and \eqref{eq:rigid:TT}.
\end{remark}
\section{Rigid $\sigma$}
Let $\sigma^{ab}_{cd}\in L$
$(a, b, c, d\in X)$
satisfying
\eqref{eq:rigid:TT}
for any $a, b, c, d\in X$,
and we write $\sigma=(\sigma^{ab}_{cd})_{a, b, c, d\in X}$.
This section deals with a property of $\sigma$ that makes
the algebra $A_\sigma$ a Hopf algebroid.
\begin{definition}
$\sigma$ is called rigid
(cf. \cite[Section 4.5]{etingof}),
iff,
for any $a, b\in X$,
there exist $x_{ab}, y_{ab}\in A_\sigma$
such that
\begin{align*}
\sum_{c\in X}((L^{-1})_{cb}+I_\sigma)x_{ac}
&=
\sum_{c\in X}x_{cb}((L^{-1})_{ac}+I_\sigma)
\\
&=
\sum_{c\in X}(L_{cb}+I_\sigma)y_{ac}
\\
&=
\sum_{c\in X}y_{cb}(L_{ac}+I_\sigma)
\\
&=
\delta_{ab}1_{A_\sigma}.
\end{align*}
\end{definition}
\begin{proposition}\label{prop:rigid:main}
The following conditions are equivalent\/$:$
\begin{enumerate}
\item[$(1)$]
$\sigma$ is rigid\/$;$
\item[$(2)$]
There exists a $k$-algebra anti-automorphism
$S: A_\sigma\to A_\sigma$
such that
\begin{align*}
&S(f\otimes 1_L+I_\sigma)=1_L\otimes f+I_\sigma,
S(1_L\otimes f+I_\sigma)=f\otimes 1_L+I_\sigma
\quad(\forall f\in L),
\\
&
S(L_{ab}+I_\sigma)=(L^{-1})_{ab}+I_\sigma
\quad(\forall a, b\in X).
\end{align*}
\end{enumerate}
\end{proposition}
\begin{proof}
The condition (2) implies (1),
if
we set
$x_{ab}=S((L^{-1})_{ab}+I_\sigma)$
and 
$y_{ab}=S^{-1}(L_{ab}+I_\sigma)$
$(a, b\in X)$.

Next we show that the condition (1) induces (2).
Because $k\langle Gen\rangle$ is free,
there uniquely exists a $k$-algebra homomorphism
$\overline{S}: k\langle Gen\rangle\to A_\sigma^{op}$
such that
\begin{equation}\label{eq:rigid:overS}
\begin{cases}
\overline{S}(l\otimes l')=l'\otimes l+I_\sigma
&(l, l'\in L(=L^{op}));
\\
\overline{S}(L_{ab})=(L^{-1})_{ab}+I_\sigma
&(a, b\in X);
\\
\overline{S}((L^{-1})_{ab})=x_{ab}
&(a, b\in X).
\end{cases}
\end{equation}
This $\overline{S}: k\langle Gen\rangle\to A_\sigma$
is a $k$-algebra anti-homomorphism.

We claim that the following is true.
\begin{claim}\label{claim:rigid:overS}
$\overline{S}(I_\sigma)=\{ 0\}$.
\end{claim}
Assuming this claim for the moment,
we complete the proof.
This claim immediately induces
a $k$-algebra anti-homomorphism
$S: A_\sigma\to A_\sigma$
defined by 
$S(a+I_\sigma)=\overline{S}(a)$
$(a\in k\langle Gen\rangle)$.
This $S$ is the desired one.
\end{proof}
\begin{proof}[Proof of Claim $\ref{claim:rigid:overS}$]
It is sufficient to show $\overline{S}(a)=0$
for any generator of the two-sided ideal $I_\sigma$.
We give the proof only for the case that $a$ is the generator
(4).

Let $x', y', x'', y''\in X$.
From the generator (4) of $I_\sigma$,
\begin{align}\nonumber
0_{A_\sigma}=&\sum_{a, b, c, d\in X}
((L^{-1})_{x''a}+I_\sigma)((L^{-1})_{y''c}+I_\sigma)\times
\\\nonumber
&\quad\times
(\sum_{x, y\in X}(\sigma_{ac}^{xy}\otimes 1_L)L_{yd}L_{xb}+I_\sigma
-\sum_{x, y\in X}(1_L\otimes\sigma_{xy}^{bd})L_{cy}L_{cx}+I_\sigma)
\times
\\\label{eq:rigid:claim}
&\quad\times
((L^{-1})_{bx'}+I_\sigma)((L^{-1})_{dy'}+I_\sigma).
\end{align}
On account of 
the generators (2) and (3) of $I_\sigma$,
the right-hand-side of \eqref{eq:rigid:claim} is 
\begin{align*}
&\sum_{a, c\in X}
((L^{-1})_{x''a}+I_\sigma)((L^{-1})_{y''c}+I_\sigma)
(\sigma_{ac}^{x'y'}\otimes 1_L+I_\sigma)
\\
-&
\sum_{b, d\in X}
((L^{-1})_{bx'}+I_\sigma)((L^{-1})_{dy'}+I_\sigma)\times
\\
&\quad\times
(1_L\otimes T_{\deg(d)}T_{\deg(b)}T_{\deg(x'')^{-1}}
T_{\deg(y'')^{-1}}(\sigma_{x''y''}^{bd})+I_\sigma).
\end{align*}
It follows from \eqref{eq:rigid:TT}
that
$T_{\deg(d)}T_{\deg(b)}T_{\deg(x'')^{-1}}
T_{\deg(y'')^{-1}}(\sigma_{x''y''}^{bd})
=\sigma_{x''y''}^{bd}$,
and \eqref{eq:rigid:claim} is
\begin{align*}
0_{A_\sigma}=&\sum_{a, c\in X}
((L^{-1})_{x''a}+I_\sigma)((L^{-1})_{y''c}+I_\sigma)
(\sigma_{ac}^{x'y'}\otimes 1_L+I_\sigma)
\\
&-\sum_{b, d\in X}
((L^{-1})_{bx'}+I_\sigma)((L^{-1})_{dy'}+I_\sigma)
(1_L\otimes \sigma_{x''y''}^{bd}+I_\sigma).
\end{align*}
Therefore, by virtue of \eqref{eq:rigid:overS},
\begin{align*}
&\overline{S}(\sum_{x, y\in X}(\sigma^{xy}_{ac}\otimes1_L)L_{yd}L_{xb}
-
\sum_{x, y\in X}(1_L\otimes\sigma^{bd}_{xy})L_{cy}L_{ax})
\\
=&
\sum_{x, y\in X}\overline{S}(L_{xb})\overline{S}(L_{yd})
\overline{S}(\sigma^{xy}_{ac}\otimes1_L)
-
\sum_{x, y\in X}\overline{S}(L_{ax})\overline{S}(L_{cy})
\overline{S}(1_L\otimes\sigma^{bd}_{xy})
\\
=&
0_{A_\sigma},
\end{align*}
and the proof is complete.
\end{proof}
\section{Hopf algebroid $A_\sigma$}
In this section,
we introduce the notion of 
the Hopf algebroid
and proceed with the study of the left and right bialgebroid
$A_\sigma$.

Let $A_L=(A, L, s_L, t_L, \Delta_L, \pi_L)$
be a left bialgebroid, together with 
an anti-automorphism $S$ of the ring $A$,
and let $L'$ be a ring isomorphic to
the opposite ring $L^{op}$.
We will denote by $\nu: L^{op}\to L'$ this isomorphism.

According to \eqref{eq:leftbi:action},
the ring $A$ has the left $L$-module structure
denoted by ${}_LA$
and
the right $L$-module structure denoted by $A_L$:
\[
{}_LA: l\cdot a=s_L(l)a;
A_L: a\cdot l=t_L(l)a
\quad(l\in L, a\in A).
\]
Moreover,
the ring $A$ has left and right $L'$-module structures
written by ${}^{L'}\!A$ and $A^{L'}$ respectively:
\[
{}^{L'}\!A:
r\cdot a=as_L(\nu^{-1}(r));
A^{L'}:
a\cdot r=aS(s_L(\nu^{-1}(r)))
\quad(r\in L', a\in A).
\]

If the maps $S$, $s_L$, and $t_L$
satisfy
\begin{equation}\label{eq:hopfalg:Sts}
S\circ t_L=s_L,
\end{equation}
then
there uniquely exists
a $\Z$-module map
$S_{A\otimes_LA}: A_L\otimes {}_LA\to A^{L'}\otimes{}^{L'}\!A$
such that
$S_{A\otimes_LA}(a\otimes b)=S(b)\otimes S(a)$
$(a, b\in A)$.

From \eqref{eq:hopfalg:Sts},
$S(a_{(1)})a_{(2)}$ makes sense.
Here,
$\Delta_L(a)=a_{(1)}\otimes a_{(2)}$ is Sweedler's notation. 
If the maps $S$, $s_L$, $t_L$,
and $\pi_L$
satisfy \eqref{eq:hopfalg:Sts}
and
\begin{equation}\label{eq:hopfalg:Stpi}
S(a_{(1)})a_{(2)}=t_L\circ\pi_L\circ S(a)
\quad(\forall a\in A),
\end{equation}
then
there uniquely exists
a $\Z$-module map
$S_{A\otimes_{L'}A}: A^{L'}\otimes {}^{L'}\!A\to A_L\otimes{}_LA$
such that
$S_{A\otimes_{L'}A}(a\otimes b)=S(b)\otimes S(a)$
$(a, b\in A)$.

We write $\Delta_{L'}$ for $S_{A\otimes_LA}\circ\Delta_L\circ S^{-1}$.
\begin{definition}
A pair $(A_L, S)$
of a left bialgebroid and an anti-automorphism
$S$ of the ring A 
satisfying
\eqref{eq:hopfalg:Sts},
\eqref{eq:hopfalg:Stpi},
and
\[
(\Delta_L\otimes\mathrm{id}_A)\circ\Delta_{L'}
=(\mathrm{id}_A\otimes\Delta_{L'})\circ\Delta_L,
(\Delta_{L'}\otimes\mathrm{id}_A)\circ\Delta_L
=(\mathrm{id}_A\otimes\Delta_L)\circ\Delta_{L'}
\]
is a Hopf algebroid,
iff
there exists the inverse
$S_{A\otimes_{L'}A}^{-1}$
of $S_{A\otimes_{L'}A}$
such that
\[
S_{A\otimes_LA}\circ\Delta_L\circ S^{-1}=
S_{A\otimes_{L'}A}^{-1}\circ\Delta_L\circ S.
\]
\end{definition} 
If $\sigma$ satisfies $\eqref{eq:rightbi:rhoT}$,
$A_\sigma$ is a left bialgebroid and a right bialgebroid
on account of Proposition \ref{prop:rightbi:rightleft};
moreover,
\eqref{eq:rigid:TT} holds
(See the proof of Proposition \ref{prop:rightbi:rightleft}).
\begin{theorem}
The algebra $A_\sigma$
with the $k$-algebra anti-automorphism $S$
in Proposition $\ref{prop:rigid:main}$ $(2)$
is a Hopf algebroid
for a rigid $\sigma$ satisfying $\eqref{eq:rightbi:rhoT}$.
\end{theorem}
The proof of this theorem is 
similar to that of Theorem 3.9 in
\cite{s2016}.
\section{Sufficient conditions for rigidity}
In this section,
we continue the study of the rigid $\sigma$ in Section 4.
We will write $\tilde{\sigma}^{ab}_{cd}=
T_{\deg(d)^{-1}}(\sigma^{ab}_{cd})\in L$
for $a, b, c, d\in X$.
\begin{theorem}\label{thm:suffcr:rigid}
Under the following conditions $(1)$--$(5)$,
$\sigma$ satisfying
$\eqref{eq:rigid:TT}$
is rigid.
\begin{enumerate}
\item[$(1)$]
For any $a, b, c, d\in X$, there exists
$i_{\ast}(\tilde{\sigma})_{cd}^{ab}\in L$
such that
$\displaystyle
\sum_{a, b\in X}i_{\ast}(\tilde{\sigma})_{zb}^{wa}\;\tilde{\sigma}_{ax}^{by}=
\sum_{a, b\in X}\tilde{\sigma}_{zb}^{wa}\;i_{\ast}(\tilde{\sigma})_{ax}^{by}=
\delta_{wx}\delta_{yz}1_L$\/$;$
\item[$(2)$]
For $a, b\in X$,
we will write
$\displaystyle Q_{ab}:=\sum_{u\in X}i_{\ast}(\tilde{\sigma})_{ua}^{ub}\in L$.
Then there exists $Q_{ab}^{-1}\in L$
such that
$\sum_{b\in X}Q_{ab}Q_{bc}^{-1}=\delta_{ac}1_L$\/$;$
\item[$(3)$]
For $a, b\in X$,
let $Q'_{ab}$
denote the element $\displaystyle \sum_{u\in X}i_{\ast}(\tilde{\sigma})_{au}^{bu}\in L$.
Then
there exists ${Q'}_{ab}^{-1}\in L$
such that
$\sum_{b\in X}{Q'}_{bc}^{-1}Q'_{ab}=\delta_{ac}1_L$\/$;$
\item[$(4)$]
For $a, b\in X$,
we will set $\displaystyle Q''_{ab}:=
\sum_{u\in X}T_{\deg(b)}(i_{\ast}(\tilde{\sigma})_{au}^{bu})\in L$. 
Then there exists ${Q''}_{ab}^{-1}\in L$
such that
$\sum_{b\in X}{Q''}_{ab}^{-1}Q''_{bc}=\delta_{ac}1_L$\/$;$
\item[$(5)$]
For $a, b\in X$,
let us denote by 
$Q'''_{ab}$
the element
$\displaystyle \sum_{u\in X}
T_{\deg(a)^{-1}}(i_{\ast}(\tilde{\sigma})_{ub}^{ua})\in L$.
Then there exists ${Q'''}_{ab}^{-1}\in L$
such that
$\sum_{b\in X}Q'''_{ab}{Q'''}_{bc}^{-1}
=\delta_{ac}1_L$.
\end{enumerate}
\end{theorem}
From the assumptions (1)--(5) in this theorem,
\begin{align*}
x_{ab}&=\sum_{c, d\in X}(Q_{ac}\otimes Q_{db}^{-1})L_{cd}+I_\sigma
\\
&=\sum_{c, d\in X}({Q''}_{ac}^{-1}\otimes Q''_{db})L_{cd}+I_\sigma,
\\
y_{ab}&=\sum_{c, d\in X}({Q'}_{db}^{-1}\otimes Q'_{ac})(L^{-1})_{cd}
+I_\sigma
\\
&=\sum_{c, d\in X}(Q'''_{bd}\otimes {Q'''}_{ca}^{-1})(L^{-1})_{cd}
+I_\sigma
\end{align*}
for any $a, b\in X$.

This theorem is proved in much the same way as 
Theorem 4.1 in
\cite{s2016}.
\section{Examples}
According to
the thesis \cite{otsuto},
this last section deals with examples of
$\sigma=(\sigma^{ab}_{cd})_{a, b, c, d\in X}$
satisfying 
\eqref{eq:rightbi:rhoT}
and
(1)--(5) in Theorem \ref{thm:suffcr:rigid}
(See also the paper \cite[Section 4]{s2016}).
Therefore, by means of $\sigma$,
we can construct the Hopf algebroid $A_\sigma$.

\begin{definition}
A non-empty set $QG$ with a binary operation is a quasigroup,
iff it satisfies:
\begin{enumerate}
\item
there uniquely exists the element $a\in QG$
such that $ab=c$
for any $b, c\in QG$;
\item
there uniquely exists the element $b\in QG$
such that $ab=c$
for any $a, c\in QG$.
\end{enumerate}
The unique element $a\in QG$ in the condition (1) will be denoted by
$c/b$,
and
the unique $b\in QG$ in (2) will be denoted by
$a\backslash c$.
\end{definition}
For example, any group is a quasigroup.
However, the quasigroup is not always associative.
\begin{example}
Let $QG_5=\{ 0, 1, 2, 3, 4\}$ denote the set of five elements,
together with the binary operation in Table 1
\cite[Section 4]{s2016}.
\begin{table}
\caption{The binary operation on $QG_5$}
\begin{center}
\begin{tabular}{c|ccccc}
\hline
&$0$&$1$&$2$&$3$&$4$\\
\hline
$0$&$4$&$3$&$2$&$1$&$0$\\
$1$&$3$&$1$&$0$&$2$&$4$\\
$2$&$0$&$2$&$3$&$4$&$1$\\
$3$&$1$&$0$&$4$&$3$&$2$\\
$4$&$2$&$4$&$1$&$0$&$3$\\
\hline
\end{tabular}
\end{center}
\end{table}
Here $02=2$.
Because each element of $QG_5$ appears 
once and only once in each row
and in each column of Table 1,
$QG_5$ with this binary operation is a quasigroup
\cite[Theorem I.1.3]{pflugfelder}.
This is not associative,
since
$(12)3=1\ne 4=1(23)$.
\end{example}

Let $QG$ be a finite quasigroup with at least two elements, 
$M$ a set isomorphic to the set $QG$,
and 
$\mu: M\times M\times M\to M$ a ternary operation on the set $M$.
We write $\pi: QG\to M$
for a bijection between the sets $QG$ and $M$.
The requirement on the ternary operation $\mu$ is the following:
\begin{enumerate}
\item[(QG1)]
$\mu(a, \mu(a, b, c), \mu(\mu(a, b, c), c, d))
=\mu(a, b, \mu(b, c, d))$
for any $a, b, c, d\in M$;
\item[(QG2)]
$\mu(\mu(a, b, c), c, d)=
\mu(\mu(a, b, \mu(b, c, d)), \mu(b, c, d), d)$
for any $a, b, c, d\in M$;
\item[(QG3)]
for any $b, c, d\in M$,
there uniquely exists $a\in M$ such that $\mu(a, b, c)=d$;
\item[(QG4)]
for any $a, c, d\in M$,
there uniquely exists $b\in M$ such that $\mu(a, b, c)=d$;
\item[(QG5)]
for any $a, b, d\in M$,
there uniquely exists $c\in M$ such that $\mu(a, b, c)=d$.
\end{enumerate}

For any finite quasigroup $QG$ satisfying $|QG|>1$
and
an abelian group $M$ isomorphic to the set $QG$,
the ternary operation
$\mu$ on $M$ defined by
$\mu(a, b, c)=a-b+c$
$(a, b, c\in M)$
enjoys all of the above conditions.

We set $H:=QG$
and $X:=QG$.
Let $G$ denote the opposite group 
of the symmetric group on the set $H$.
For $a\in QG$,
$\deg(a)\in G$ is given by
$\lambda\deg(a)=\lambda a$
$(\lambda\in H=QG)$.

Let $R$ be a $k$-algebra.
Here $k$ is a commutative ring.
We define the map
$\sigma^{ab}_{cd}: H\to R$
by
\[
\sigma^{ab}_{cd}(\lambda)=\begin{cases}
1_R,&\mbox{if }c=\pi^{-1}(\mu(\pi(\lambda), \pi(\lambda b),
\pi((\lambda b)a)))\backslash((\lambda b)a)
\\
&
\mbox{and }
d=\lambda\backslash\pi^{-1}(\mu(\pi(\lambda), \pi(\lambda b),
\pi((\lambda b)a)));
\\
0_R,&\mbox{otherwise}.
\end{cases}
\]
Let $L$ denote the $k$-algebra
consisting of maps from the set $H$ to $R$.
The product of this $k$-algebra $L$ is defined by
$(fg)(\lambda)=f(\lambda)g(\lambda)$
$(f, g\in L, \lambda\in H)$.
It follows that $L\cong R^{|H|}$
as $k$-algebras.
\begin{theorem}
This $\sigma=(\sigma^{ab}_{cd})_{a, b, c, d\in X}$
is rigid.
\end{theorem}
The proof of this theorem is straightforward;
in fact, we can show that $\sigma$ satisfies
\eqref{eq:rightbi:rhoT}
and
(1)--(5) in Theorem \ref{thm:suffcr:rigid}.
\begin{corollary}
$A_\sigma$ is a Hopf algebroid.
\end{corollary}
If $R$ is not separable with an idempotent Frobenius system
(For the definition, see Introduction),
then so is $L(\cong R^{|H|})$.
As a result,
$A_\sigma$ is not a weak Hopf algebra
on account of \cite[Theorem 5.1]{schauenburg2003}.
\section*{Acknowledgments}
The work was supported in part by JSPS KAKENHI Grant Number
JP17K05187.

\end{document}